\documentclass[11pt]{amsart}
\usepackage[english]{babel}
\usepackage{amsmath}
\usepackage{amsthm}
\usepackage{amssymb}
\usepackage{amscd}
\usepackage{latexsym}
\usepackage[all]{xy}
\usepackage[latin1]{inputenc}
\xyoption{all}
\newtheorem{teor}{Theorem}[section]
\newtheorem{defi}[teor]{Definition}
\newtheorem{lema}[teor]{Lemma}
\newtheorem{prop}[teor]{Proposition}
\newtheorem{cor}[teor]{Corollary}
\newtheorem{rem}[teor]{Remark}

\usepackage{anysize}
\marginsize{2.5cm}{2.5cm}{2.5cm}{3.4cm}
\def\Hom{\mathop{\rm Hom}\nolimits}
\def\ov{\overline}

\def\Mod{\textrm{-}{\rm Mod}}

\def\Ext {\mathop{\rm Ext}\nolimits}

\def\Ker {\mathop{\rm Ker}\nolimits}
\def\Phantq{\mathbf{Phant}_R(\mathcal A_2)}
\def\Idealq{\mathcal I(\mathcal A_2)}
\def\Phant{\mathbf{Phant}(R)}
\def\Im {\mathop{\rm Im}\nolimits}

\begin{document}
\title{Covering Ideals of Morphisms and Module Representations of the quiver $\mathbb{A}_2$}
\author{Sergio Estrada}
\address{Departamento de Matem\'atica Aplicada, Universidad de Murcia, 30100 Murcia, Spain}
\email{sestrada@um.es}

\author{Pedro A. Guil Asensio}
\address{Departamento de Matem\'aticas, Universidad de Murcia, 30100 Murcia, Spain}
\email{paguil@um.es}

\author{Furuzan Ozbek}
\address{Department of Mathematics, University of Kentucky, 40506 Lexington, Kentucky, USA}
\email{furuzanozbek@uky.edu}

\thanks{The first two authors have been partially supported by the DGI
MTM2010-20940-C02-02 and by
the
Fundaci\'on Seneca 04555/GERM/06. }
\date{}


{

\begin{abstract}
Sufficient conditions for an ideal $\mathcal I$ in $R\Mod$ to be covering are proved. This allows to obtain an alternative proof of the existence of phantom covers of modules. Our approach is inspired by an extension of the standard deconstructibility techniques used in Approximation Theory.
\end{abstract}

\maketitle

\section{Introduction}
Ideal Approximation Theory has been recently introduced by Herzog, Fu, Guil Asensio and Torrecillas in \cite{FGHT}. This theory establishes an extension to ideals of morphisms in general exact categories of the usual theory of covers and envelopes by modules which was independently initiated by Enochs and Auslander-Smal\o\ and that has attracted a wide interest in the last years (see e.g. the monographs \cite{BR,EO, GT,Xu} for a detailed exposition). This extension of the theory has experienced an increasing development in the last year (see e.g. \cite{FGHT2,FH,FH2}).
The main reason of this growing interest is that this proposed extension covers significative examples of covers and envelopes of modules which were not included in the original theory because they need to be stated in terms of morphisms instead of objects. This is the case, for instance, of the existence of phantom covers of modules obtained in \cite{Herzog1}, or the existence of almost split morphisms in the Abelian category $A$-mod of finitely presented modules over Artin algebras (see e.g. \cite{ARS}).  

This introduction of an approximation theory for ideals of morphisms naturally leads to the statement of existence theorems for covers and envelopes associated to these ideals, similar to the ones obtained in the classical theory of covers and envelopes. 
Following this approach, it has been recently proved in \cite{Ozbek} that if $\mathcal I$ is an ideal in a module category $R\Mod$ generated by a set (see \cite{Ozbek} for a definition of this concept), 
then $\mathcal{I}$ is a precovering ideal in the sense of \cite{FGHT}. 
This result extends to this new setting a classical existence theorem of precovers and preenvelopes associated to cotorsion pairs of objects (see e.g. \cite{GT}). However, the Ideal Approximation Theory proposed in \cite{FGHT} seems to be more complicated than the classical one and therefore, the existence of phantom covers of modules proved in \cite{Herzog1}, and which represent a central example in the theory, cannot be infered from this result.

The main motivation of this paper is to introduce a new criterion for the existence of covers of modules associated to ideals of morphisms. Namely, we prove in Section 3  sufficient conditions to ensure  that an ideal in $R\Mod$ is covering, in the sense that any $R$-module has an $\mathcal I$-cover (cf. \cite{FGHT}). And these sufficient conditions are satisfied by the ideal of phantom morphisms in $R\Mod$ introduced in \cite{Herzog1}. The key observation in our paper is that one may identify ideals $\mathcal I$ of \emph{morphisms} in $R\Mod$ with certain classes $\Idealq$ of \emph{objects} in the Grothendieck category $\mathcal A_2$ of all representations by left $R$-modules of the quiver $\mathbb{A}_2:\bullet\to \bullet$. This allows to apply the general existence theorems developed in the literature for (pre)covers of modules with respect to a class of objects. Namely, we show:

\medskip\par\noindent
{\bf Theorem \ref{prin}.} Let $\mathcal I$ be an ideal in $R\Mod$ closed under direct limits. If $\mathcal I$ is the closure under direct limits of a set of morphisms $\mathcal I_0\subseteq\mathcal I$, then $\mathcal I$ is a covering ideal.
\smallskip

As we have just observed, our approach has the main  advantage of relating existence theorems for covering ideals of morphisms to the corresponding ones for covering classes of objects in the Grothendieck category $\mathcal A_2$. But  it also has an important drawback: this identification does not respect extensions in the sense that, if $f,g$ are two homomorphisms in $R\Mod$, 
then $\Ext^1_R(f,g)$ (see e.g. \cite{FGHT} for its definition) is different from the extension group of the associated representations in $\mathcal A_2$, 
$\Ext^1_{\mathcal A_2}(f,g)$. 
As a consequence, the class of objects $\Idealq$ in $\mathcal A_2$ associated to any given  ideal of morphisms $\mathcal I$ in $R\Mod$ is not closed under extensions. Therefore, some classical results, as the so-called Wakamatsu Lemmas (which have been recently extended to this new setup in \cite{FH2}), cannot be proved by means of this approach (see Remark \ref{noext} for a more detailed explanation).

We devote our last section to apply our results to  phantom ideals of morphisms, in the sense of Benson \cite{Be} (see also \cite{Gna, Herzog1}). We recall that the notion of phantom morphism has its origin in the study of morphisms between CW complexes \cite{MG}. It was later extended  by Neeman \cite{Neeman} to triangulated categories and used by Benson and Gnacadja to study the stable category associated to $k[G]\Mod$, where $G$ is a finite group and $k$ is a field whose characteristic divides the order of $G$ (see e.g. \cite{Be, Gna}). Christensen and Strickland in \cite{CS} were the first authors who realized the close connection between phantom morphisms and Purity Theory (see e.g. \cite{C}). And this relation has later shown to be essential to obtain a general proof of the existence of phantom covers of modules \cite{Herzog1,Herzog2}. We show in Section 4 that the ideal $\Phant$ of phantom morphisms in $R\Mod$  satisfies the above existence theorem. This allows us to get an alternative proof of \cite[Theorem 7]{Herzog1}:

\medskip\par\noindent
{\bf Corollary \ref{phantomcover}.} Every left $R$-module has a surjective phantom cover.
\smallskip

We also obtain  a new proof of the result obtained in \cite{Herzog1} that the kernels of these phantom covers are always pure injective modules (see also \cite{Herzog2}).

We would finally like to stress that our Theorem \ref{main} also shows that the ideal $\Phant$ admits a {\em nice} filtration by $\leq \kappa$-presented phantom morphisms, for some infinite cardinal $\kappa$. Classes of objects satisfying this property are usually called in the literature \emph{deconstructible}. Deconstructible classes of objects $\mathcal F$ which are closed under transfinite extensions play a central role in the classical approximation theory (see e.g. \cite{S,Tr}), as they provide special $\mathcal F$-precovers and special $\mathcal F^{\perp}$-preenvelopes of modules. However, Remark \ref{noext} shows that the class of objects $\Idealq$ that we are associating to any ideal $\mathcal I$ of morphisms in $R\Mod$ is never closed under the above type of filtrations, unless $\mathcal I=\Hom$. This fact suggests that a new notion of deconstrucible classes of morphisms may be needed for Ideal Approximation Theory.

\section{Preliminaries and notation.}

\subsection{Ideals in categories and phantom morphisms.}

Let $R$ be an associative ring with identity and let us denote by $R\Mod$ the category of unitary left modules over $R$. Following \cite{FGHT}, we will call an additive subbifunctor of the bifunctor $\Hom:R\Mod^ {op}\times R\Mod\to \textrm{Ab}$ an \emph{ideal} $\mathcal I$ of $R\Mod$. This means that given $R$-homomorphisms $f,g,h,t$ with $f,g\in\mathcal I$, we get that $f+g\in \mathcal I$ and $h\circ f\circ t\in \mathcal I$, whenever they are defined.

An important instance of ideal in $R\Mod$ is the ideal of phantom morphisms considered in \cite{Herzog1}. Let us recall that a morphism $f:M\to N$ of left $R$-modules is called a \emph{phantom morphism} if for each morphism $g:L\to M$, with $L$ a finitely presented left $R$-module, the composition $f\circ g$ factors through a (finitely presented) projective module. Equivalently, for each left $R$-module $A$, the functor $\Ext^1(f,A)$ maps $\Ext^1(N,A)$ inside the subgroup ${\rm PExt}^1(M,A)$ of $\Ext^1(M,A)$ consisting of all pure-exact sequences.
It is straightforward to check that the class of all phantom morphisms forms an ideal, which we will denote by $\Phant$. 

The next definition is the natural extension to ideals of morphisms of the usual notions of (pre)covers and (pre)envelopes with respect to a class of modules (see \cite{FGHT}).

\begin{defi}
Let $\mathcal I\subseteq \Hom_R$ be an ideal in $R\Mod$ and $M\in R\Mod$. An $\mathcal I$-\emph{precover} of $M$ is a morphism $i:I\to M$ in $\mathcal I$ such that any morphism $i':I'\to M$ in $\mathcal I$ factors through $i$. I.e., there exists a commutative triangle

$$\xymatrix{   I'\ar@{.>}[d]\ar[dr]^{i'}   &  \\
               I\ar[r]^{i}  & M.   }$$
An $\mathcal I$-precover $i:I\to M$ is said to be an $\mathcal I$-cover if every map $j$ that completes the diagram
$$\xymatrix{   I\ar[d]_j\ar[dr]^{i}   &  \\
               I\ar[r]^{i}  & M    }$$
is necessarily an automorphism. An ideal $\mathcal I$ is said to be \emph{(pre)covering} if every $R$-module $M$ admits an $\mathcal I$-(pre)cover. 

$\mathcal I$-\emph{(pre)envelopes} and \emph{(pre)\-en\-ve\-loping} ideals are defined dually.
\end{defi}

\medskip
\subsection{The category $\mathcal A_2$}

Let us denote by $\mathbb{A}_2$ the quiver with two vertices $v_1,v_2$ and an edge $a:v_1\to v_2$. This may be thought as a small category.
Let us consider the category $\mathcal A_2=(\mathbb{A}_2,R\Mod)$ of {\it all representations}  of the quiver $\mathbb{A}_2$ by left $R$-modules. 
That is, the category of all covariant functors from $\mathbb A_2$ to $R\Mod$. Note that an object $\overline{M}$ of $\mathcal A_2$ is just a morphism $\overline{M}\equiv M_1\stackrel{f}{\to}M_2$ in $R\Mod$. Whereas a morphism in $\mathcal A_2$ from $\ov{M}\equiv M_1\stackrel{f}{\to}M_2$ to $\ov{N}\equiv N_1\stackrel{g}{\to}N_2$ is a natural transformation; that is, a pair of morphisms $(d,s)$ in $R\Mod$ for which the diagram
$$\xymatrix{M_1\ar[r]^f\ar[d]_{d} & M_2 \ar[d]^{s} \\ N_1\ar[r]_g & N_2}$$ is commutative. The category $\mathcal A_2$ is Grothendieck and the representations $ R\stackrel{1_R}{\to}R$ and $ 0\stackrel{0}{\to}R$ are projective representations that generate $\mathcal A_2$. It is also known that $\ov{P}\equiv P_1\stackrel{f}{\to}P_2$ is a projective representation if, and only if, $P_1$ and $P_2$ are projective $R$-modules and $f$ is a splitting monomorphism. And $\ov{P}$ is flat provided that $P_1$ and $P_2$ are flat $R$-modules and $f$ is a pure monomorphism (see \cite{EEsProj,EnLoyBlas}).

We can associate to any ideal $\mathcal I$ in $R\Mod$, the class of objects $\Idealq$ in $\mathcal{A}_2$  
consisting of those  representations $\overline{M}\equiv M_1\stackrel{f}{\to}M_2$ in $\Idealq$ such that  $f\in\mathcal I$.
In particular, we will denote by $\Phantq$ the class of all phantom morphisms in $R\Mod$, considered as a class of representations in $\mathcal A_2$. Hence, $\overline{M}\equiv M_1\stackrel{f}{\to}M_2$ belongs to $\Phantq$ if and only if $f$ is a phantom morphism in $R\Mod$. It is clear, from the definition of phantom morphism, that flat representations of $\mathbb A_2$ belong to  $\Phantq$. This means that, if $\ov{F}\equiv F_1\to F_2$ is a representation in which $F_1$ is a flat $R$-module, then $\ov{F}\in \Phantq$. In particular, $\Phantq$ contains a projective generator of $\mathcal A_2$.

As $\mathcal A_2$ is a Grothendieck category, we have the usual notions of \emph{(pre)covers} and \emph{(pre)envelopes} with respect to a class of representations $\mathcal F$ in $\mathcal A_2$. Namely,
if $\mathcal F$ is a class of objects in $\mathcal A_2$ and $\ov{M}\in \mathcal A_2$, an $\mathcal F$-precover ($\mathcal F$-preenvelope) of $\ov{M}$ is
a morphism $\ov{F}\stackrel{\varphi}{\rightarrow} \ov{M}$
(resp., $\ov{M}\stackrel{\varphi}{\rightarrow} \ov{F}$) with $\ov{F}\in\mathcal F$, such that
$\Hom (\ov{F'},\ov{F})\rightarrow \Hom(\ov{F'}, \ov{M})\rightarrow 0$ is exact  (resp., $\Hom
(\ov{M},\ov{F'})\rightarrow  \Hom(\ov{F},\ov{F'})\rightarrow 0$ is  exact), for every
$\ov{F'}\in\mathcal F$. If, moreover, any $f:\ov{F}\rightarrow \ov{F}$ such that
$\varphi\circ f=\varphi$ (resp. $f\circ \varphi=\varphi$) is an
automorphism, then $\varphi$ is called an $\mathcal F$-cover
(resp., $\mathcal F$-envelope). Note that $\mathcal F$-covers (resp.
$\mathcal F$-envelopes) are unique up to isomorphism whenever they do exist. So we will talk of {\em the $\mathcal F$-cover} (resp., {\em the $\mathcal F$-envelope} ) of a representation $\ov{M}$ with
the understanding that this uniqueness is up to isomorphisms. We will say that a class $\mathcal F$ of representations in $\mathcal A_2$ is (pre)covering ((pre)enveloping)) if every $\ov{M}\in \mathcal A_2$ admits an $\mathcal F$-(pre)cover (resp., $\mathcal F$-(pre)envelope).

Throughout this paper, all rings will be associative rings with identity
element and all modules will be unitary left modules. We will denote by $R\Mod$ the category of all left $R$-modules. We refer to \cite{FGHT, EO, GT, Xu} for any undefined notion used along this text.

\section{A Sufficient Condition for Covering Ideals}

Let us begin this section by introducing the following definition which will be needed to state the main result of this section.

\begin{defi} Let $\mathcal I$ be an ideal in $R\Mod$. We will say that $\mathcal I$ is {\em closed under direct limits} if for any morphism $\{f_i:M_i\rightarrow N_i\}_I$ between directed systems of morphisms $\{g_{ij}:M_i\rightarrow M_j\}_{i\leq j}$ and $\{h_{ij}:N_i\rightarrow N_j\}_{i\leq j}$, satisfying that $f_i\in\mathcal I$ for every $i\in \mathcal I$, the induced morphism $\underrightarrow{\rm lim}\,f:\underrightarrow{\rm lim}M_i\rightarrow \underrightarrow{\rm lim}N_i$ also belongs to $\mathcal I$. 

And we will say that the ideal $\mathcal I$ is the {\em closure under direct limits} of a set of morphisms $\mathcal I_0\subseteq\mathcal I$ if there exists a set $\mathcal I_0\subseteq \mathcal I$ such that any morphism $f\in\mathcal I$ can be obtained in the above fashion from a morphism $\{f_i\}_I$ of direct systems with each $f_i\in \mathcal I_0$. 
\end{defi}
  
We can now prove our promised criterion for the existence of covering ideals of morphisms.

\begin{teor}\label{prin}
Let $\mathcal I$ be an ideal in $R\Mod$ closed under direct limits. If $\mathcal I$ is the closure under direct limits of a set of morphisms $\mathcal I_0\subseteq\mathcal I$, then $\mathcal I$ is a covering ideal of $R\Mod$.
\end{teor}

\begin{proof} Note that our hypothesis implies that $\Idealq$ is the closure under direct limits in the  Grothendieck category  $\mathcal A_2$ of the set $\mathcal I_0(\mathcal A_2)$. 
Hence $\Idealq$ is a covering class in $\mathcal A_2$ by \cite[Theorem 3.2]{Bashir}. 

Let us first show that this implies that  
$\mathcal{I}$ 
is a precovering ideal. Let $M$ be an $R$-module and 
$\ov{F}\equiv F_1 \stackrel{\varphi}{\to} F_2$, an $\Idealq$-cover of $\ov{1}_{M}\equiv M \stackrel{1_M}{\to} M$. 
Then there exist morphisms $s,t$ such that the following diagram commutes,
$$\xymatrix{F_1 \ar[r]^{s} \ar[d]_{\varphi} & M \ar@{=}[d] \\ 
F_2 \ar[r]^{t} & M} $$
Now, given any $G\stackrel{\psi}{\to}M$ in $\Idealq$, there exist morphisms $\alpha, \beta$ which make the following diagram commutative
$$\xymatrix{ & G \ar@{.>}[dl]_{\alpha} \ar[dr]^{\psi} \ar[dd]^>>>>>>{\psi}& \\
F_1 \ar[rr]^>>>>>>>>>>{s} \ar[dd]_{\varphi}& & M \ar@{=}[dd] \\
&M \ar@{.>}[dl]_{\beta} \ar@{=}[dr]& \\
F_2 \ar[rr]^{t} & & M}$$
since $\ov{F}$ is an $\Idealq$-precover. Note that this implies that $s=t \circ \varphi $ belongs to $\mathcal{I}$, since $\mathcal I$ is an ideal, and that the top triangle commutes. Hence, $s$ is an $\mathcal{I}$-precover of $M$. This shows that $\mathcal{I}$ is a precovering ideal.

Finally, \cite[Theorem 2.2.12]{Xu} (see \cite[Lemmas 2.1,2.2,2.3]{E} for the original argument) can be easily adapted to the framework of ideals of morphisms to deduce that if an ideal $\mathcal{I}$ of $R\Mod$ is closed under direct limits and $M$ has an $\mathcal{I}$-precover, then it also has an $\mathcal{I}$-cover. This finishes the proof.
\end{proof}

As noted in \cite[Theorem 3.3]{Bashir}, the above arguments can be easily carried out to deduce the following stronger result under the assumption of Vop\v enka's Principle (see e.g. \cite[Chapter 6]{AR}).

\begin{teor}
(Assume Vop\v enka's Principle) Any ideal $\mathcal I$ of $R\Mod$ closed under direct limits is covering.
\end{teor}

\begin{rem}\label{noext}
We would like to stress that the usual version of Wakamatsu's Lemma \cite[Lemma 2.1.1]{Xu} in $\mathcal A_2$ cannot be used to infer from Theorem \ref{prin} that the kernel of an $\Idealq$-cover is an object in $\Idealq^ {\perp}=\{\ov{M}\in \mathcal A_2|\ \Ext^1_{\mathcal A_2}(\ov{F},\ov{M})=0,\ \forall \ov{F}\in \Idealq\}$. The reason is that the class of representations $\Idealq$ of $\mathcal A_2$ which we have associated to any ideal  $\mathcal I$ of $R\Mod$ is not closed under extensions unless the ideal $\mathcal I=\Hom$. Note that if $f:A\rightarrow B$ is a morphism in $R\Mod$ such that $f\notin\mathcal I$ and we consider the following commutative diagram of splitting sequences, 
$$\xymatrix{
0\ar[r] & A\ar[d]^0\ar[r]^{\tau_1} & A\oplus A \ar[d]^{t_1f\pi_2}\ar[r]^{\pi_2} &A\ar[d]^0\ar[r] &0\\
0\ar[r] & B\ar[r]_{t_1} & B\oplus B \ar[r]_{p_2} &B\ar[r] &0}$$
then $0:A\rightarrow B$ belongs to $\mathcal I$ but $t_1\circ f\circ\pi_2\notin\mathcal I$ since $f\notin\mathcal I$. This means that the representation $A\oplus A\stackrel{t_1f\pi_2}{\longrightarrow} B\oplus B$ is an extension of the representation $0:A\rightarrow B$ by itself which does not belong to $\Idealq$.
 
\end{rem}

\section{Filtering Phantom Morphisms}
Throughout this last section we will focus on the case in which the considered ideal $\mathcal I$ is the class $\Phant$ of all phantom morphisms in $R\Mod$. Our aim is to prove that each phantom morphism admits a {\em nice} filtration by {\em small} phantom morphism, and deduce from it that the ideal $\Phant$ satisfies the hypothesis of Theorem \ref{prin}.

Let ${\mathcal A}$ be a
Grothendieck category and $\lambda$, an ordinal number. Recall that a linearly ordered directed system of morphisms in $\mathcal A$, $\{f_{\alpha\beta}:A_\alpha\rightarrow A_\beta|\,\alpha\leq\beta<\lambda\}$, is called continuous if $A_\gamma=\underset{\alpha\leq\beta<\gamma}{\underrightarrow{\lim}}\, f_{\alpha\beta}$ for each ordinal limit $\gamma\leq\lambda$. Note that, in particular, this means that $A_0=0$. 

A linearly ordered directed system of morphisms is called a {\em continuous directed union} if all morphisms in the system are monomorphisms.
Given a class of objects $\mathcal L$ in $\mathcal A$, an
object $A$ of $\mathcal A$ is said to be \emph{$\mathcal L$-filtered} (or a direct transfinite extension of $\mathcal L$) if  $ A =
\displaystyle \lim_{ \rightarrow } \: A_{ \alpha } $ for a
 linearly ordered continuous directed union $ \{f_{\alpha\beta}: A_{ \alpha}\rightarrow A_\beta|\alpha\leq\beta<\lambda\} $  
satisfying that, 
for each $ \alpha<\lambda$, $\mathrm{Coker}(f_{\alpha,\alpha+1})$
is isomorphic to an element of $ {\mathcal L} $.
Finally, a class of objects $\mathcal L$  of a Grothendieck category $\mathcal A$ is called \emph{deconstructible} if there exists a set $\mathcal L_0\subseteq\mathcal L$ such that each object $L\in \mathcal L$ is $\mathcal L_0$-filtered.
This notion, which was introduced by Eklof \cite{Ek}, has  shown to play a central role in the classical aproximation theory by objects (see e.g. \cite{S,Tr}).

\begin{defi}\label{decons}
We will say that the ideal $\mathcal I$ in $R\Mod$ is deconstructible when the class $\Idealq$ is deconstructible in $\mathcal A_2$.
\end{defi}

\begin{rem}
The previous notion of deconstructible class is weaker than others which appear in the literature (see for instance \cite[Definition 1.4]{S}), where it is required that {\em every} $\mathcal L_0$-filtered object belongs to the class $\mathcal L$ (see also \cite[Definition 1.5]{Tr}).
\end{rem}

Our goal in this section is to show in Theorem \ref{main} that the class of phantom morphisms is deconstructible, in the sense of the above definition. To achieve this aim, we will start with the following technical lemma.

\begin{lema}
Let $\kappa\geq|R|$ be an infinite cardinal number. Let $\ov{M}\equiv M_1\stackrel{h}{\to} M_2$ be a representation in $\mathcal A_2$ and let $X_1\subseteq M_1$,
$X_2\subseteq M_2$ be any two subsets with $|X_1|, |X_2|\leq
\kappa$. Then there exists a subrepresentation $\ov{M'}$ of $\ov{M}$ of the form $\ov{M'}\equiv M_1'\to M_2'$, such that:

\begin{itemize}
\item [i) ] $M_1'\subseteq M_1$ and $M_2'\subseteq M_2$ are pure submodules.
\item [ii) ]$X_1\subseteq M_1'$, $X_2\subseteq M_2'$.
\item [iii) ]$|M_1'|,|M_2'|\leq \kappa$.
\end{itemize}\label{subrep}
\end{lema}
\begin{proof} Let $\kappa$, $X_1$, $X_2$ be as given in the statement of the lemma. By \cite[Lemma 5.3.2]{EO}, we conclude that there is a pure submodule $M_1'$ of $M_1$ such that $X_1 \subset M_1'$ and $|M_1'| \leq {\rm max}(|R|, |X_1|)$. Hence, $|M_1'| \leq \kappa$. Using the same argument, we can find another pure submodule $M_2'$ of $M_2$, containing $X_2$, such that $h(M_1') \subset M_2'$ and $|M_2'| \leq {\rm max}(|R|, |M_1'|)$. In particular, $|M_2'|\leq \kappa$. We get then the following commutative diagram
$$\xymatrix{ M_1' \ar@{^{(}->}[r] \ar[d]_{h|_{M_1} } & M_1 \ar[d]^{h} \\
M_2' \ar@{^{(}->}[r] & M_2 }$$
which is the desired subrepresentation of $\ov{M}$ satisfying the three requiered properties.
\end{proof}

It is well known that the Grothendieck category $\mathcal A_2$ is locally finitely presented. Therefore we can state the following definition of purity in $\mathcal A_2$, which also reflects the notion of purity contained in the above lemma.

\begin{defi}
A subrepresentation $\ov{N}\equiv N_1\to N_2$ of $\ov{M}\equiv M_1\to M_2$ is called \emph{pure} provided that $N_1$ and $N_2$ are pure submodules of $M_1$ and $M_2$, respectively.
\end{defi}

\begin{lema}\label{phantomsub} 
Let $\kappa \geq |R|$ be an infinite cardinal number, $\ov{F}\equiv F_1\to F_2$, a representation in $\Phantq$, and $X_1\subseteq F_1,X_2\subseteq F_2$ two subsets with $|X_1|,|X_2|\leq \kappa$. Then there is a pure phantom subrepresentation $\ov{S}\equiv  S_1 \to S_2$ of $\ov{F}$ such that $|S_1|, |S_2| \leq \kappa$ and $X_1\subseteq S_1$ and $X_2\subseteq S_2$.
\end{lema}

\begin{proof} By Lemma~\ref{subrep}, we may find a pure subrepresentation of $\ov{F}\equiv  F_1 \to F_2$ which is small as desired
$$\xymatrix{  S_1 \ar@{^{(}->}[r] \ar[d]_{\varphi\mid_{S_1}} & F_1 \ar[d]^{\varphi} \\
T \ar@{^{(}->}[r] & F_2 }.$$

Now, we must transform this subrepresentation into a phantom one. Let $L$ be any finitely presented $R$-module and $h: L \to S_1$, any homomorphism. As $\ov{F}\in \Phantq$, there exists a finitely presented projective $R$-module $P_h$ (whose cardinality is therefore, bounded by $\kappa$) such that $L \to F_2$ factors through it
$$\xymatrix{L \ar[d]_{h} \ar@/^1pc/[ddrr]& & \\
S_1 \ar@{^{(}->}[r] & F_1  \ar[dd]^{\varphi} & \\
& & P_h \ar[dl]^{f_{h}} \\
T \ar@{^{(}->}[r] & F_2 & } $$
Define 
$$\tilde{T}= T + \sum\{{\rm Im}(f_{h})|\, L\text{ is a finitely presented module and }h\in{\rm Hom}(L,S_1)\}$$ 
and note that $|\tilde{T}| \leq \kappa$. $\tilde{T}$ is not necessarily a pure submodule of $F_2$ but, by Lemma~\ref{subrep}, we may find a pure submodule $\tilde{T} \subset S_2 \subset F_2$ such that $|S_2|\leq \kappa$. Then, by construction, for any arbitrary finitely presented $R$-module $L$ and any $h:L \to S_1$, there exists a projective $R$-module $P$ such that the following diagram commutes
$$\xymatrix{L \ar[d]_{h} \ar@/^1pc/[ddrr]& & \\
S_1 \ar[dd]_{\varphi\mid_{S_1}} \ar@{^{(}->}[r] & F_1  \ar[dd]^{\varphi} & \\
& & P \ar[dll]_{f_{h}} \\
S_2 \ar@{^{(}->}[r] & F_2 & } $$
since $\Im(f_h )\subset S_2$. This means that $L \to S_2$ factors through $P$. And hence, $\xymatrix{S_1 \ar[r]^{\varphi\mid_{S_1}} & S_2}$ belongs to $\Phantq$. 
\end{proof}

\begin{lema}\label{clos}
The class $\Phantq$ is closed under direct limits.
\end{lema}
\begin{proof} Let $\{f_i:N_i\rightarrow M_i\}_I$ be a morphism between two directed systems of morphisms $\{g_{ij}:N_i\rightarrow N_j\}_{i\leq j}$ and $\{h_{ij}:M_i\rightarrow M_j\}_{i\leq j}$ in $R\Mod$ such that $f_i:N_i\rightarrow M_i$ is a phantom morphism, for each $i\in I$. 
And let $\underrightarrow{\lim}f_i:\underrightarrow{\lim}N_i\rightarrow\underrightarrow{\lim}M_i$ be the induced morphism in the direct limits. We must show that $\underrightarrow{\lim}f_i$ is also a phantom morphism.

Let $L$ be any finitely presented left $R$-module. We have to prove that for any morphism $L \to \displaystyle \lim_{ \rightarrow } \: N_i$ the composition $L \to \displaystyle \lim_{ \rightarrow } \: N_i \to \displaystyle \lim_{ \rightarrow } \: M_i$ factors through a projective $R$-module. But, as $L$ is finitely presented, we have an  isomorphism
$$\Hom(L,\lim_{ \rightarrow } \: N_i) \cong \lim_{ \rightarrow } \Hom(L, N_i).$$

Hence, given a morphism $L \to \displaystyle \lim_{ \rightarrow } \: N_i$, there is a $j \in I$ such that the following diagram can be completed
$$\xymatrix{ & L \ar@{.>}^{\circlearrowleft}[dl] \ar[d] \\
N_j \ar[r] \ar[d] & \displaystyle\lim_{ \rightarrow } \: N_i \ar[d] \\
M_j \ar[r] &  \displaystyle\lim_{ \rightarrow } \: M_i }$$

Now, as $f_j\in\Phant$, we conclude that the composition $L \to N_j \to M_j$ factors through a projective $R$-module, say $P$
$$\xymatrix{ & & L \ar[dll] \ar[d] \\
N_j \ar[rr] \ar[dd] \ar[dr]& & \displaystyle\lim_{ \rightarrow } \: N_i \ar[dd] \\
& P \ar[dl]& \\
M_j \ar[rr]& & \displaystyle\lim_{ \rightarrow } \: M_i} $$
and then, by diagram chasing, we conclude that the morphism $L \to \displaystyle \lim_{ \rightarrow } \: M_i$ factors through $P$ as well. \end{proof}

\begin{defi} Let $\ov{M}\equiv M_1\to M_2$ be a representation in $\mathcal A_2$. We define the {\em cardinality of} $\ov{M}$, $|\ov{M}|$, as the cardinality of the disjoint union $M_1\sqcup M_2$.
\end{defi}

We can now state:
 
\begin{teor}\label{main}
The class of phantom morphisms  is deconstructible.
\end{teor}
\begin{proof} Let $\ov{F}\equiv M \stackrel{f}{\to} N$ be an arbitrary representation in $\Phantq$ and let us fix an infinite cardinal number $\kappa \geq |R|$. 
We are going to define, for each ordinal $\beta$, pure subrepresentations $\ov{F}_\beta\equiv M_\beta\stackrel{f_\beta}{\to}N_\beta$ of 
$\ov{F}$, and pure embeddings $g_{\alpha\beta}:M_\alpha\to M_\beta$ and 
$h_{\alpha\beta}:N_\alpha\to N_\beta$, for each ordinal $\alpha\leq\beta$, in such a way that they satisfy: 

\begin{enumerate}
\item $\ov{F}_0$ is the zero representation $0\stackrel{0}{\to} 0$.

\item For any ordinal $\nu$, the families $\{g_{\alpha\beta}:M_\alpha\to M_\beta\}_{\alpha\leq\beta<\nu}$ and 
$\{h_{\alpha\beta}:N_\alpha\to N_\beta\}_{\alpha\leq\beta<\nu}$ are continuous directed unions of morphisms; and the family of morphisms $\{f_\alpha:M_\alpha\to N_\alpha\}_{\alpha<\nu}$ is a morphism between both directed systems. I.e., 
$f_\beta\circ g_{\alpha\beta}=h_{\alpha\beta}\circ f_\alpha$, for each $\alpha\leq\beta<\nu$.

\item The induced morphism $$\frac{f_{\alpha+1}}{f_{\alpha}}:\frac{M_{\alpha+1}}{M_\alpha}\to \frac{N_{\alpha+1}}{N_\alpha},$$ is phantom, for each ordinal $\alpha$.

\item $|M_{\beta+1}/M_{\beta}|,|N_{\beta+1}/N_{\beta}|\leq\kappa$,
for each ordinal $\beta$.

\item If $\ov{F}_\beta\neq \ov{F}$, then the morphism $(g_{\beta,\beta+1},h_{\beta,\beta+1}):\ov{F}_\beta\to \ov{F}_{\beta+1}$ is not an isomorphism.
\end{enumerate}

Note that, in particular, this implies that there exists an ordinal $\lambda$ such that $\ov{F}_\lambda=\ov{F}$ because $M,N$ are sets. Therefore, if we are able to make the above construction, we will have shown that $\Phant$ is $\mathcal S_\kappa$-deconstructible, where $\mathcal S_\kappa$ is the set of all isomorphism classes of phantom morphisms or cardinality bounded by $\kappa$.  

We are going to make this construction by transfinite induction on $\alpha$. Let us set $\ov{F}_0\equiv 0\stackrel{0}{\to}0$. Fix now an arbitrary ordinal $\gamma>0$ and assume that we have constructed $\ov{F}_\beta$, $g_{\alpha\beta}$ and $h_{\alpha\beta}$ with the desired properties, for each $\alpha\leq\beta<\gamma$. We are going to construct then $\ov{F}_\gamma$ and $g_{\alpha\gamma}$ and $h_{\alpha\gamma}$, for each $\alpha\leq\gamma$ satisfying the above properties. 

\medskip
{\bf Case 1.}  Assume that $\gamma$ is a limit ordinal. Set then $M_\gamma=\underset{\alpha\leq\beta<\gamma}{\underrightarrow{\lim}}\, g_{\alpha\beta}$, 
$N_\gamma=\underset{\alpha\leq\beta<\gamma}{\underrightarrow{\lim}}\,h_{\alpha\beta}$\, and 
$f_\gamma:M_\gamma\to N_\gamma$, the morphism induced by the $\{f_\alpha:M_\alpha\to N_\alpha\}_{\alpha<\gamma}$. Let us call, for each $\alpha<\gamma$, $g_\alpha:M_\alpha\to M_\gamma$ and $h_\alpha:N_\alpha\to N_\gamma$ the structural morphisms of the direct limits. Set then $g_{\alpha\gamma}=g_\alpha$ and $h_{\alpha\gamma}=h_\alpha$, for each $\alpha<\beta$. It is straightforward to show that $\ov{F}_\gamma\equiv M_\gamma\stackrel{f_\gamma}{\to}N_\gamma$, $g_{\alpha\gamma}$ and $h_{\alpha\gamma}$, for $\alpha\leq\gamma$ (where $g_{\gamma\gamma}=1_{M_{\gamma}}$ and $h_{\gamma\gamma}=1_{N_{\gamma}}$) satisfy the desired properties.
\medskip

{\bf Case 2}. Assume now that $\gamma=\mu+1$ is a successor ordinal. 
If $\ov{F}_\mu=\ov{F}$, then we set 
$\ov{F}_\gamma=\ov{F}$, 
$g_{\mu\gamma}=1_M$, 
$h_{\mu\gamma}=1_N$, and $g_{\alpha\gamma}=g_{\mu\gamma}\circ g_{\alpha\mu}$ and
$h_{\alpha\gamma}=h_{\mu\gamma}\circ h_{\alpha\mu}$, for each $\alpha<\mu$, and $g_{\gamma\gamma}=1_{M}$ and $h_{\gamma\gamma}=1_{N}$.

Otherwise, the induced representation $\ov{F}/\ov{F_\mu}\equiv M/M_\mu\stackrel{f/f_{\mu}}{\longrightarrow}N/N_\mu$ is nonzero, in the sense that $M/M_\mu$ or $N/N_\mu$ are not zero. We claim that $f/f_\mu$ is a phantom morphism. As $\ov{F_\mu}$ is a pure subrepresentation of $\ov{F}$, we deduce that, in particular, the short exact sequence,
$$\xymatrix{0 \ar[r]& M_\mu \ar[r]& M \ar[r] & M/M_\mu \ar[r]& 0}$$
is pure. Then, for any finitely presented $R$-module $L$ and any morphism $L \to M/M_\mu$, there exists a morphism $L\to M$ which completes the following commutative diagram
$$\xymatrix{ &&&L \ar@{.>}[dl] \ar[d] & \\
0 \ar[r] & M_\mu \ar[r] \ar[d] & M \ar[r] \ar[d] & M/M_\mu \ar[r] \ar[d] & 0\\
0 \ar[r] & N_\mu \ar[r] & N \ar[r] & N/N_\mu \ar[r]& 0} $$
 Moreover, as $\ov{F} \in \Phantq$, the composition $L \to M \to N$ factors through a projective $R$-module, say $P$,
$$\xymatrix{ && L \ar[dll] \ar[d] \ar[ddl] & \\
M \ar[rr] \ar[dd] && M/M_\mu \ar[r] \ar[dd] & 0\\
&P \ar[dl]&& \\
N \ar[rr] & & N/N_\mu \ar[r]& 0} $$
This shows that  $L \to N/N_\mu$ also factors through $P$ and hence, we deduce that  $\ov{F}/\ov{F_\mu} \in \Phantq$.

Now, as $\ov{F}\neq \ov{F_\mu}$, there exists an $x \in  M\sqcup N\setminus M_\mu\sqcup N_\mu$. By Lemma~\ref{phantomsub}, there exists a pure subrepresentation $\ov{F'}/\ov{F_\mu}$ of $\ov{F}/\ov{F_{\mu}}$, where $\ov{F'}:M'\stackrel{f'}{\longrightarrow} N',$   such that $x \in M'\sqcup N'$, $x\notin  M_{\mu}\sqcup N_{\mu}$ and $|\frac{M'}{M_\mu}\sqcup \frac{N'}{N_\mu}| \leq \kappa$. Let us denote by $g_\mu':M_\mu\to M'$ and $h_\mu':N_\mu\to N'$ the embedding maps such that $h_\mu'\circ f_\mu=f'\circ g_\mu'$. It is then clear that $g_\mu'$ and $h_\mu'$ are pure embeddings, because $\ov{F_\mu}$ is a pure subrepresentation of $\ov{F}$ by our induction hypothesis. Moreover, as $\ov{F'}/\ov{F_\mu}$ is a pure subrepresentation of $\ov{F}/\ov{F_{\mu}}$ and $\ov{F_\mu}$ is a pure subrepresentation of $\ov{F}$, it is immediate to deduce that $\ov{F'}$ is a pure subrepresentaton of $\ov{F}$.

Finally, let us set 
$\ov{F}_\gamma=\ov{F'}$, 
$g_{\mu\gamma}=g_\mu'$, 
$h_{\mu\gamma}=h_\mu'$, and $g_{\alpha\gamma}=g_{\mu\gamma}\circ g_{\alpha\mu}$ and
$h_{\alpha\gamma}=h_{\mu\gamma}\circ h_{\alpha\mu}$, for each $\alpha<\mu$, and $g_{\gamma\gamma}=1_{M'}$ and $h_{\gamma\gamma}=1_{N'}$ It is easy to check that $\ov{F}_\gamma$, $g_{\alpha\gamma}$ and $h_{\alpha\gamma}$, for $\alpha\leq\gamma$ satisfy the desired properties.
\end{proof}

Recall that a representation $\ov{M}\equiv M_1\to M_2$ in $\mathcal A_2$ is of \emph{type $\kappa$} (for $\kappa$ an infinite cardinal number),
if each of the $R$-modules $M_1$ and $M_2$ are generated at most by $\kappa$ elements.

\begin{cor}\label{colim}
There exists an infinite cardinal $\kappa$ such that every representation in $\Phantq$ is the directed union of its pure subrepresentations in $\Phantq$ of type $\kappa$.
\end{cor}

\begin{cor}\label{phantomcover}
Every module has a surjective phantom cover.
\end{cor}
\begin{proof}
It follows from Lemma \ref{clos} and Corollary \ref{colim} that the class $\Phantq$ satisfies the conditions of Theorem \ref{prin}. Finally, phantom covers are surjective because the class $\Phantq$  contains a projective generator of $\mathcal A_2$.
\end{proof}

\begin{rem} Note that Theorem \ref{main} shows that every phantom map is $\mathcal S$-filtered (where $\mathcal S$ is a set of representatives of $\ov{M}\in \Phantq$ with cardinality less than or equal to $\kappa$). However, the converse is not true in view of Remark\ \ref{noext}.

Deconstructible classes of objects play a central role in the classical aproximation theory by objects. The reason is that any  deconstructible class $\mathcal F$ (in the sense of Definition \ref{decons})  in a Grothendieck category  $\mathcal A$, which is closed under $\mathcal F$-filtered objects, gives rise to the existence of special $\mathcal F^{\perp}$-preenvelopes and $\mathcal F$-precovers of objects in $\mathcal A$ \cite[Theorem 2.5 and 2.6]{EEGO} (indeed, special $\mathcal F$-precovers whenever we assume, in addition, that $\mathcal F$ contains a generator of the category). 
However, as Remark \ref{noext} shows, the class $\Phantq$ is not closed under extensions in $\mathcal A_2$ and thus, we cannot use the fact that 
$\Phant$ 
is deconstructible to infer the existence of special $\Phant$-precovers and special 
$\Phant^\perp$-preenvelopes in 
$R\Mod$.

\end{rem}

We are going to close this paper by showing that the kernel of a phantom cover is always a pure injective module. First we need to prove the following lemma, which is of independent interest.

\begin{lema}\label{indep}
Let $\phi:M\to N$ be a phantom epimorphism, with kernel $u:K\to M$ and let $v:K\to K'$ be a pure monomorphism. Let us consider the pushout along $u,v$:
 $$\xymatrix{ & 0\ar[d] & 0\ar[d] &  &\\
0\ar[r] & K\ar[d]_v\ar[r]^{u} & M \ar[r]^{\phi} \ar[d]^{v'}\ar[r] &N\ar@{=}[d]\ar[r] &0\\
0\ar[r] & K'\ar[d]_{\pi}\ar[r]_{u'} & X\ar[d]^{\pi'} \ar[r]_{\phi'} &N\ar[r] &0 \\ & K/K'\ar@{=}[r]\ar[d] & X/M\ar[d] &  &\\& 0 & 0 &  & }$$
Then $\phi'$ is also a phantom map.
\end{lema}
\begin{proof}
Let $F$ be a finitely presented module and consider a morphism $f:F\to X$. As the short exact sequence 
$$0\to K\stackrel{v}{\longrightarrow}K'\stackrel{\pi}{\longrightarrow}K/K'\to 0$$ 
is pure, there will exist a morphism $g:F\to K'$ such that $\pi\circ g=\pi'\circ f$. Then $\pi\circ g=\pi'\circ u'\circ g$ and thus, $\Im(f-u'g)\subseteq \Ker(\pi')=M$. Therefore, there exists a unique $h:F\to M$ such that $v'h=f-u'g$.
Now $$ \phi h=\phi'v'h=\phi'f-\phi'u'g'=\phi'f,$$ where the last equality holds because $\phi'u'=0$. Now, as $\phi$ is phantom, $\phi h$ factors through a projective module. This shows that $\phi' f$ factors through a projective module and thus, $\phi'$ is phantom.

\end{proof}

\begin{prop}\label{kernel_pinj}
Let $\phi:M\to N$ be a phantom cover. Then $\Ker(\phi)$ is a pure injective module.
\end{prop}
\begin{proof}
Let $K=\Ker(\phi)$ and $u:K\to M$, the inclusion. We must show that $K$ is pure injective. So let $u: K\to X$ be a pure monomorphism. We want to see that it admits a retract. As phantom covers are surjective, it follows from Lemma \ref{indep} that the pushout along $u$ and $v$,
$$\xymatrix{0\ar[r] & K\ar[d]_v\ar[r]^{u} & M \ar[r]^{\phi} \ar[d]^{v'}\ar[r] &N\ar@{=}[d]\ar[r] &0\\
0\ar[r] & X\ar[r]_{u'} & M' \ar[r]_{\phi'} &N\ar[r] &0   }$$
gives a phantom morphism $\phi':M'\to N$.
 This leads to commutative triangles:
$$\xymatrix{   M\ar[d]_{v'}\ar[dr]^{\phi}   &  \\
               M'\ar[r]^{\phi'}\ar[d]_{t}  & N \\
               M\ar[ur]_{\phi}       }$$ where $t:M'\to M$ comes from the fact that $\phi'$ is phantom and $\phi$ is a phantom (pre)cover. Let us denote by $w=t\upharpoonright_{X}$ the restriction of $t$ to $X\to K$.   As $\phi$ is a cover, the morphism $tv':M\to M$ is an automorphism. Hence, the restriction $w\circ v:K\to K$ is an automorphism of $K$ and therefore, $r=(wv)^{-1}w$ is the desired retract of $v$.
\end{proof}

\end{document}